\documentclass{amsart}
\usepackage[utf8]{inputenc}
\usepackage{amsmath,amssymb,amsthm, tikz}
\usepackage{xcolor}
\usepackage{hyperref}
\usepackage{comment}

\newtheorem{theorem}{Theorem}%[section]
\newtheorem{lemma}[theorem]{Lemma}

\newtheorem{question}[theorem]{Question}
\theoremstyle{definition}
\newtheorem{definition}[theorem]{Definition}
\newtheorem{example}[theorem]{Example}

\newtheorem{remark}[theorem]{Remark}

\newcommand{\co}{\colon\thinspace}
\newcommand{\eps}{\varepsilon}

\DeclareMathOperator{\Area}{Area}
\DeclareMathOperator{\Vol}{Vol}
\DeclareMathOperator{\diam}{diam}

\DeclareMathOperator{\Hom}{Hom}
\DeclareMathOperator{\sys}{sys}
\DeclareMathOperator{\cut}{cut}
\DeclareMathOperator{\polylog}{polylog}

\begin{document}

\title{Systolic almost-rigidity modulo 2}
\author{Hannah Alpert}
\address{Auburn University, 221 Parker Hall, Auburn, AL 36849}
\email{hcalpert@auburn.edu}
\author{Alexey Balitskiy}
\address{Institute for Advanced Study, 1 Einstein Drive, Princeton, NJ 08540, and University of Luxembourg, 6 avenue de la Fonte, L-4364 Esch-sur-Alzette, Luxembourg}
\email{alexey.balitskiy@uni.lu}
\author{Larry Guth}
\address{Massachusetts Institute of Technology, 77 Massachusetts Ave, Cambridge, MA 02139}
\email{lguth@math.mit.edu}

\begin{abstract}
%For any integer $n \ge 2$ and any $\eps > 0$ there is a constant $c = c(n,\eps) > 0$ such that the following systolic inequality holds for 
No power law systolic freedom is possible for the product of mod~$2$ systoles of dimension $1$ and codimension $1$. This means that any closed $n$-dimensional Riemannian manifold $M$ of bounded local geometry obeys the following systolic inequality: the product of its mod~$2$ systoles of dimensions $1$ and $n-1$ is bounded from above by $c(n,\eps) \Vol(M)^{1+\eps}$, if finite (if $H_1(M; \mathbb{Z}/2)$ is non-trivial).
\end{abstract}

\maketitle

\section{Results}
\label{sec:intro}

Let $M$ be a finite simplicial complex whose geometric realization is a closed $n$-dimensional manifold. Denote by $\Vol(M)$ the number of $n$-simplices in $M$. The \emph{$k$-dimensional mod~$2$ systole} $\sys_k(M)$ is defined as the minimal number of $k$-simplices in a simplicial $k$-cycle representing a nontrivial homology class in $H_k(M; \mathbb{Z}/2)$. If $H_k(M; \mathbb{Z}/2)$ is trivial, then $\sys_k(M) = +\infty$ by convention.

The following systolic inequality (which goes back to Loewner in the case of two-dimensional torus) seems natural to conjecture: 
\begin{equation*}
\label{eq:sys}
\sys_k(M) \cdot \sys_{n-k}(M) \le c_n \sys_n(M) = c_n \Vol(M),  \tag{$\star$}
\end{equation*}
where $H_k(M)$ is assumed to be non-trivial to make the left-hand side finite. 
The phenomenon of \emph{systolic freedom} is the failure of~\eqref{eq:sys} for $n\ge 3$. For the \emph{integral} systoles (defined similarly to the mod~$2$ systoles, but with $\mathbb{Z}$ instead of $\mathbb{Z}/2$), this phenomenon was pioneered by Gromov~\cite{gromov1996systoles}. In the integral case, there are numerous counterexamples~\cite{babenko1998systolic, katz1999volume} to~\eqref{eq:sys} in all dimensions $n \ge 3$, even if the constant $c_n$ is allowed to depend on the topological type of $M$. For the mod~$2$ systoles, the same question is much more subtle, as the mod~$2$ systoles are often much smaller than the integral ones. It was Freedman~\cite{freedman1999z2, freedman2002z2} who constructed the first counterexamples to \eqref{eq:sys} over $\mathbb{Z}/2$, with $k=1$ and all $n \ge 3$. In his family of examples, the \emph{systolic ratio} $\frac{\sys_1(M) \cdot \sys_{n-1}(M)}{\Vol(M)}$ has the magnitude of $\sqrt{\log(\Vol(M))}$; these are examples of \emph{polylog systolic freedom}. We prove that they are (almost) tight in the sense that \emph{no power law systolic freedom} can be observed for $k=1$.

\begin{theorem}
\label{thm:sys-first}
For each $\eps > 0$, there is a constant $c = c(n,\eps)$ such that any triangulated $n$-dimensional manifold $M$ with non-trivial first homology $H_1(M; \mathbb{Z}/2)$ obeys the following systolic estimate:
\[
\sys_1(M) \cdot \sys_{n-1}(M) \le c \Vol(M)^{1+\eps}.
\]
\end{theorem}

This contrasts with the case of $\mathbb{Z}$ coefficients, in which there is \emph{no almost-rigidity} as in Theorem~\ref{thm:sys-first}: Gromov's examples already exhibit \emph{power law systolic freedom}, and there are examples of \emph{exponential systolic freedom}~\cite{pittet1997systoles}.

As for systolic inequalities over $\mathbb{Z}/2$ with $k\ge2$, there is an estimate $\sys_k(M) \cdot \sys_{n-k}(M) \le c_n \Vol(M)^2$, following from the trivial observation that the number of $k$-simplices in $M$ is at most ${n+1 \choose k+1} \Vol(M)$. The outstanding result of~\cite{freedman2021building} yields families of $11$-dimensional triangulated manifolds in which $\sys_4(M) \cdot \sys_7(M)$ grows faster than $\Vol(M)^{2-\eps}$. The idea behind this result owes a lot to the relation between the systolic freedom over $\mathbb{Z}/2$ and quantum error correction codes. Having a cellulated manifold exhibiting systolic freedom at the level of $k\textsuperscript{th}$ homology, one can consider the three-term portion of its cellular chain complex around the $k\textsuperscript{th}$ term, and use it to build a quantum code with error correction properties depending on the systolic freedom of the manifold (see~\cite{freedman2002z2} for the details). The idea of~\cite{freedman2021building} is to ``reverse-engineer'' manifolds starting from the breakthrough quantum error correction codes of~\cite{hastings2021fiber, panteleev2021quantum}. 
In terms of this quantum analogy, Theorem~\ref{thm:sys-first} says that for $k=1$, manifolds do not produce any outstanding codes. With our methods, we cannot say anything for $k = 2, 3$.

\subsection*{Systolic inequalities in the discrete setting}
We will prove a discrete systolic estimate for general simplicial complexes rather than triangulated manifolds. The systoles will be substituted by certain larger systolic invariants, making the inequality stronger. 

\begin{itemize}
    \item For a non-zero homology class $a \in H_k(M; \mathbb{Z}/2)$, define $\sys_a(M)$ as the minimal number of $k$-simplices in a simplicial $k$-cycle representing $a$.
    \item For a non-zero cohomology class $\alpha \in H^k(M; \mathbb{Z}/2)$, define $\sys^\alpha(M)$ as the minimal number of $k$-simplices in a simplicial $k$-cycle detected by $\alpha$.
    \item For a non-zero cohomology class $\alpha \in H^k(M; \mathbb{Z}/2)$, define $\cut^\alpha(M)$ as the minimal number of $(n-k)$-simplices in an $(n-k)$-dimensional subcomplex $H$ that \emph{cuts} $\alpha$ in the sense that $\alpha$ vanishes when restricted to $M\setminus H$ (in this difference $M$ and $H$ are understood as geometric realizations, not the combinatorial complexes); in other words, no singular $k$-cycle detected by $\alpha$ survives if we cut $M$ along $H$ (hence the name). 
\end{itemize}

\begin{theorem}
\label{thm:sys-discrete}
For each $\eps > 0$, there is a constant $c = c(n,\eps)$ such that any finite $n$-dimensional simplicial complex $M$ and any non-zero class $\alpha \in H^1(M; \mathbb{Z}/2)$ obey the following estimate:
\[
\sys^\alpha(M) \cdot \cut^{\alpha}(M) \le c \Vol(M)^{1+\eps}.
\]
\end{theorem}

To relate $\cut^\alpha$ with better known systolic quantities (and deduce Theorem~\ref{thm:sys-first}), we make two remarks.

\begin{enumerate}
    \item If we know that there is a non-trivial $(n-1)$-cycle inside the minimizing complex $H$ from the definition of $\cut^\alpha$, then we get a systolic estimate. This can be achieved by assuming that there is a cohomology class $\beta \in H^{n-1}(M; \mathbb{Z}/2)$ such that $\alpha \smile \beta \neq 0$. The Lusternik--Schnirelmann lemma applies: since $\alpha$ vanishes on the complement of $H$, $\beta$ restricts to $H$ non-trivially. Therefore, there is a non-trivial $(n-1)$-cycle in $H$ detected by $\beta$, and Theorem~\ref{thm:sys-discrete} implies that
\[
\sys^{\alpha}(M) \cdot \sys^{\beta}(M) \le c \Vol(M)^{1+\eps}.
\]
    We comment that the inequality $\sys^{\beta}(M) \le \cut^\alpha(M)$ can be strict, even for manifolds (see Example~\ref{ex:sys<cut}). 
    \item In case when $M$ is a (triangulated) manifold, it is easy to see that any cycle representing $\alpha^* \in H_{n-1}(M; \mathbb{Z}/2)$---the Poincar\'e dual class to $\alpha$---serves as an admissible cutting complex in the definition of $\cut^\alpha$, and $\cut^\alpha(M) \le \sys_{\alpha^*}(M)$. Interestingly, the minimal $H$ turns out to be a cycle representing $\alpha^*$, as we will show in Lemma~\ref{lem:regularity}. Therefore, $\cut^\alpha(M) = \sys_{\alpha^*}(M)$, and Theorem~\ref{thm:sys-discrete}, rephrased for manifolds, says that
\[
\sys^{\alpha}(M) \cdot \sys_{\alpha_*}(M) \le c \Vol(M)^{1+\eps}.
\]
    This estimate clearly implies Theorem~\ref{thm:sys-first}.
\end{enumerate}

\subsection*{Systolic inequalities in the continuous setting}

Let $M$ be a closed $n$-dimensional Riemannian manifold. The \emph{$k$-dimensional mod~$2$ systole} is defined as the infimum of the $k$-volumes of all (piecewise smooth or Lipschitz) $k$-cycles representing nontrivial homology $H_k(M; \mathbb{Z}/2)$. 

Following Freedman, we measure the \emph{systolic freedom} of a sequence of manifolds $M_i$ with volumes $V_i = \Vol(M_i)$ and non-trivial $H_k$, as follows. We scale the Riemannian metric of each $M_i$, if needed, to make its geometry \emph{locally bounded}: this means that the sectional curvatures are between $-1$ and $1$ everywhere, and the injectivity radius is at least $1$. After that we measure the \emph{systolic ratio} $\frac{\sys_k(M_i) \cdot \sys_{n-k}(M_i)}{V_i}$ and express its growth as $i \to \infty$ as a function of $V_i$. The faster this function grows the more freedom the manifolds exhibit. A constant function corresponds to the \emph{systolic rigidity}. Freedman's examples had the growth rate of $\sqrt{\log(\cdot)}$. Our main result implies that for $k=1$ the systolic ratio cannot grow faster than $(\cdot)^{\eps}$, for arbitrarily small $\eps > 0$.

\begin{theorem}
\label{thm:sys-continuous}
For each $\eps > 0$, there is a constant $c = c(n,\eps)$ such that any closed Riemannian $n$-manifold $M$ of bounded local geometry and with non-trivial $H_1(M; \mathbb{Z}/2)$ obeys the following systolic estimate:
\[
\sys_1(M) \cdot \sys_{n-1}(M) \le c \Vol(M)^{1+\eps}.
\]
\end{theorem}

To deduce Theorem~\ref{thm:sys-continuous} from Theorem~\ref{thm:sys-first}, we triangulate $M$ carefully, and replace the continuous systoles with the discrete ones.
%The systolic freedom can be similarly measured in the discrete setting. In a simplicial complex, one defines the (discrete) $k$-dimensional systole as the minimal number of $k$-simplices in a homologically non-trivial simplicial $k$-cycle. The discrete volume is just the number of top-dimensional simplices. For a sequence of triangulated manifolds, one can similarly measure the systolic freedom via the growth rate of the discrete systolic ratio. 
It is known~\cite{cairns1934triangulation, whitehead1940c1, dyer2015riemannian, boissonnat2018delaunay, bowditch2020bilipschitz, boissonnat2023local} that a Riemannian manifold $M$ of bounded local geometry can be triangulated so that if we endow each simplex with the standard euclidean metric with edge-length $1$, then the resulting metric on $M$ is bi-Lipschitz to the original one (with the Lipschitz constant depending only on $n$). In particular, the number of $n$-simplices approximately equals $\Vol(M)$ (maybe off by a dimensional factor). For a sequence of Riemannian manifolds of bounded local geometry, we have two ways to measure its systolic freedom: via the growth rate of the continuous or discrete systolic ratio. It turns out they have the same magnitude~\cite[Theorem~1.1.1]{freedman2021building}. The non-obvious part of this claim is that a systolic cycle in the continuous setting can be efficiently approximated by a discrete one in the triangulated manifold, without increasing its volume too much. This is done via a ``Federer--Fleming'' type of argument: inside each cell of the triangulation, the cycle is projected to the boundary radially from a random point; then this is repeated for cells of lower dimensions, until the cycle is approximated by a discrete one (see~\cite{freedman2021building} for details). Therefore, Theorem~\ref{thm:sys-first} implies Theorem~\ref{thm:sys-continuous}. 
We remark that manifolds of bounded local geometry, when triangulated as above, have an additional property of bounded ``degree'', in the sense that every vertex of the triangulation is incident to a uniformly bounded number of top-dimensional simplices; this property is not needed in Theorem~\ref{thm:sys-first}, making it somewhat stronger than the continuous version.

Theorem~\ref{thm:sys-continuous} can also be strengthened in a different direction: if we stay in the continuous realm, but replace the condition of bounded local geometry by its ``macroscopic'' cousin as follows. Given $0 < v_1 < v_2$, let us say that a Riemannian manifold $M$ has \emph{$(v_1,v_2)$-bounded macroscopic geometry} if
\begin{itemize}
    \item $\sys_1(M) \ge 2$,
    \item every ball of radius $1/4$ has volume at least $v_1$,
    \item and every ball of radius $1$ has volume at most $v_2$.
\end{itemize}
These three conditions should be viewed as the analogues of a lower bound on injectivity radius, upper bound of sectional curvature, and lower bound on sectional curvature, respectively.
\begin{theorem}
\label{thm:sys-macrocontinuous}
For each $\eps, v_1, v_2 > 0$, there is a constant $c = c(\eps, v_1,v_2)$ such that any compact Riemannian $n$-manifold $M$ of $(v_1,v_2)$-bounded macroscopic geometry exhibits systolic almost-rigidity:
\[
\sys_1(M) \cdot \sys_{n-1}(M) \le c \Vol(M)^{1+\eps}.
\]
\end{theorem}
This theorem is more general than Theorem~\ref{thm:sys-continuous}. Indeed, on a manifold of bounded local geometry, the systolic bound $\sys_1 \ge 2$ follows trivially from the bound on the injectivity radius, and volumetric estimates are mere corollaries of the Rauch comparison theorem (see, e.g., \cite[Section~7.1.1]{berger2003panoramic}).

\begin{remark}
    The referee asked us about the minimal set of conditions implying almost-rigidity. Namely, what assumptions one can impose on a class of manifolds to prohibit the following scenario: there exists $\eps >0$ and a sequence of manifolds $M_i$ satisfying these assumptions, with $\Vol M_i \to \infty$, and with $\sys_1(M_i) \cdot \sys_{n-1}(M_i) > c \Vol(M_i)^{1+\eps}$? We do not know if any of the three conditions of bounded macroscopic geometry in Theorem~\ref{thm:sys-macrocontinuous} can be omitted. We remark that the condition ``every ball of radius $1/4$ has volume at least $v_1$'' is not sufficient by itself. To see that, recall that Freedman constructed a sequence of three-dimensional manifolds $M_g$ (with an integer parameter $g$, which can be arbitrarily large), which have bounded local geometry, $\diam (M_g) \sim \log g$, $\sys_1 (M_g) \sim \sqrt{\log g}$, $\sys_{2} (M_g) \sim g$, $\Vol (M_i) \sim g$. They exhibit polylog systolic freedom. Scaling down the metric on $M_g$ by a factor of $g^{1/3}/ \log g$, we obtain a sequence of manifolds $M_g'$ with volumes going to infinity, and diameters going to zero. So a lower bound on the volumes of balls of radius $1/4$ is present, but the sequence now has power law systolic freedom: $\sys_1(M_g') \cdot \sys_{2}(M_g') \sim \Vol(M_g')^{7/6}$.
\end{remark}
\begin{question}
Does it suffice to assume that $\sys_1 \ge 2$, to guarantee systolic almost-rigidity? In other words, is there a sequence of manifolds $M_i$ with $\sys_1(M_i) \ge 2$ and $\Vol M_i \to \infty$, such that $\sys_1(M_i) \cdot \sys_{n-1}(M_i) > c \Vol(M_i)^{1+\eps}$ (for some $\eps >0$)? If the bound on the $1$-systole alone is not enough, does it suffice to assume both $\sys_1 \ge 2$, and a lower bound on the volumes of balls of radius $1/4$?
\end{question}

\subsection*{Structure of the paper}

The proofs of discrete systolic estimates (mainly, of Theorem~\ref{thm:sys-discrete}) are explained in Section~\ref{sec:proof}. The core idea is the ``Schoen--Yau minimal surface'' approach, after \cite{guth2010systolic, papasoglu2020uryson}. 

Some further generalizations are considered in Section~\ref{sec:generalize}. There we deal with the product of more than two systoles, and with a ``macroscopic'' systolic estimate, implying Theorem~\ref{thm:sys-macrocontinuous}.

\subsection*{Acknowledgments.} We thank the anonymous referee whose suggestions motivated us to highlight the version of the main theorem under ``macroscopic'' assumptions. This material is based in part upon work supported by the National Science Foundation under Grant No.~DMS-1926686. Larry Guth is supported by Simons Investigator Award. %\textcolor{red}{[update]}

\section{Proofs in the discrete case}
\label{sec:proof}

Let $M$ be an $n$-dimensional simplicial complex. The distance between vertices is measured as the edge-length of the shortest path between them in the $1$-skeleton of $M$. 
\begin{definition}
\label{def:ball-sphere}
Let $x$ be a vertex of $M$, and let $r$ be a non-negative integer.
\begin{enumerate}
    \item The \emph{ball} $B(x,r)$ is the subcomplex of $M$ defined as follows. If $r=0$, $B(x,0) = \{x\}$, so assume $r \ge 1$. An inclusion-maximal $k$-simplex, $0 \le k \le n$, is in $B(x,r)$ if all its vertices are distance at most $r$ away from $x$ and at least one of the vertices is at distance less than $r$ from $x$. An arbitrary simplex is in $B(x,r)$ if it is a face of an inclusion-maximal simplex as in previous sentence.
    \item The \emph{sphere} $S(x,r)$ is the subcomplex $M$ consisting of all simplices of $B(x,r)$ with vertices at distance exactly $r$ from $x$.
\end{enumerate}
\end{definition}

This definition might seem confusing, but it essentially defines spheres and balls as the level and sublevel sets of the distance function from $x$ measured in the geometric realization of $M$ with an auxiliary piecewise-Finsler metric $F$, described as follows. For each $k \le n$, consider a $k$-simplex $\triangle^k$ in a linear $k$-dimensional space and endow it with the norm whose unit ball is the Minkowski sum $\triangle^k + (- \triangle^k)$. In particular, the distance between any disjoint faces of $\triangle^k$ is $1$. Now we define $F$ by saying that its restriction to any $k$-simplex of $M$ is a flat Finsler metric isometric to $\triangle^k$. 

From this interpretation it is seen that every sphere actually separates the inside from the outside: every curve starting in $B(x,r)$ and ending outside $B(x,r)$ must intersect $S(x,r)$.

We write $\Vol$ for the number of $n$ simplices in an $n$-complex, and $\Area$ for the number of $(n-1)$-simplices in an $(n-1)$-complex.

Everywhere below $\alpha$ denotes a non-trivial class in $H^1(M;\mathbb{Z}/2)$; $H$ denotes an $(n-1)$-dimensional subcomplex of $M$.  

\begin{lemma}[Co-area inequality]
\label{lem:coarea}
$\Vol B(x,r) \ge \sum\limits_{i=1}^{r} \Area S(x,i)$.
\end{lemma}
\begin{proof}
To each $(n-1)$-simplex in $S(x,i)$ we can assign an incident $n$-simplex in $B(x,i)$, and all those $n$-simplices are distinct. 
\end{proof}

\begin{lemma}[Vitali cover]
\label{lem:vitali}
Suppose we have a ball $B(x_i, 2\rho_i)$ at each vertex of an $(n-1)$-complex $H$ inside $M$. Then one can select a subcollection of those balls that covers $H$ but such that the corresponding balls $B(x_i, \rho_i)$ do not overlap (no $n$-simplex belongs to two balls from the subcollection).
\end{lemma}
\begin{proof}
Greedily add balls $B(x_i, \rho_i)$ to the non-overlapping subcollection, in the non-decreasing order of radii.
\end{proof}

\begin{lemma}[Curve factoring]
\label{lem:curve}
If a $1$-cycle $\ell$ lies in a ball of radius $r < \frac{\sys^\alpha - 1}{2}$, then $\alpha([\ell]) = 0$.
\end{lemma}
\begin{proof}
The cycle $\ell$ can be split as a sum of loops of length at most $2r+1$, and none of those is detected by $\alpha$.
\end{proof}

\begin{lemma}[Cut-and-paste trick]
\label{lem:cut-and-paste}
Suppose $H$ cuts $\alpha$, $x \in H$, and $r < \frac{\sys^\alpha - 1}{2}$. Then the complex $H' = H \setminus (H\cap B(x,r)) \cup S(x,r)$ cuts $\alpha$ as well. 
\end{lemma}
\begin{proof}
Every connected component of a $1$-cycle in the complement of $H'$ either avoids $H$ (and then it is not detected by $\alpha$ since $H$ cuts $\alpha$) or falls inside $B(x,r)$ (and then it is not detected by $\alpha$ by Lemma~\ref{lem:curve}).
\end{proof}

\begin{proof}[Proof of Theorem~\ref{thm:sys-discrete}]
Let $V = \Vol M$, $R = \sys^\alpha(M)$, and let $H$ be a cutting subcomplex on which $\cut^{\alpha}(M)$ is attained. %We can safely add the entire $(n-2)$-skeleton of $M$ to $H$ without increasing its area. 

Consider all singular (that is, not necessarily simplicial) loops detected by $\alpha$. There are two cases.
\begin{enumerate}
    \item Every such loop meets the $(n-2)$-skeleton of $M$. Then $H$ can be taken to be this skeleton, $\Area H = 0$, and there is nothing to prove.
    \item There is a loop avoiding the $(n-2)$-skeleton of $M$. Deform it to the $1$-skeleton of $M$. Each of the edges of the deformed loop is incident to an $n$-dimensional simplex of $M$. If these $n$-simplices are not all distinct, we can replace the loop with a shorter one (still detected by $\alpha$). Once we make it as short as possible, all incident $n$-simplices are distinct, and we conclude that $R \le V$. We use this inequality in what follows.
\end{enumerate}

The second inequality we observe is $\Area H \le (n+1)V$. Indeed, consider an $(n-1)$-simplex in $H$. If it is not incident to $n$-simplices in $M$, we can safely remove it from $H$ (but leave its boundary) preserving the cutting property of $H$. Therefore, to every $(n-1)$-simplex of $H$ we can assign an incident $n$-simplex of $M$, and it follows that $\Area H \le (n+1)V$.

If $\Area H \ge R^{1/\eps}$, then 
\[
R \cdot \Area H \le (\Area H)^{1+\eps} \le (n+1)^{1+\eps} V^{1+\eps}, 
\]
so in the rest of the argument we assume that $\Area H < R^{1/\eps}$.

For brevity, we write $v(x,r) = \Vol B(x,r)$ and $h(x,r) = \Area(H \cap B(x,r))$.

A ball centered at $x \in H$ of positive radius divisible by $2$ will be called \emph{good} if 
\[
h(x,r) \le \frac{4}{R^{1-\eps}} v(x,r/2).
\]
A sufficient condition for a ball of radius $0 < r \le R$ divisible by $4$ to be good is
\[
h(x,r) \le \frac{r}{R^{1-\eps}} h(x,r/4).
\]
Indeed, by the co-area inequality (Lemma~\ref{lem:coarea}) one has
\begin{align*}
h(x,r) &\le \frac{r}{R^{1-\eps}} \frac{1}{r/2 - r/4} \sum_{t = r/4}^{r/2-1}  h(x,t) \\
&\le \frac{r}{R^{1-\eps}} \frac{4}{r} \sum_{t = r/4}^{r/2-1}  \Area S(x,t) \\
&\le \frac{4}{R^{1-\eps}} v(x,r/2).
\end{align*}
Here we used the minimality of $H$, which implies $h(x,t) \le \Area S(x,t)$, since otherwise we can apply the cut-and-paste trick (Lemma~\ref{lem:cut-and-paste}), decreasing the area of $H$ and preserving its cutting property.

We will prove below that for every $x \in H$ there is a good ball centered at $x$. Having a cover of $H$ by good balls, we pick a Vitali subcover $\bigcup B(x_i, 2r_i) \supset H$ such that the balls $B(x_i, r_i)$ do not overlap (by Lemma~\ref{lem:vitali}). This will conclude the proof:
\[
R \cdot \Area H \le R \sum h(x_i, 2r_i) \le 4 R^\eps \sum v(x,r_i) \le 4 V^{1+\eps}.
\]

Now we show that around every $x \in H$ one can place a good ball. In the range $[R^{1-\eps}, R]$, pick a longest sequence of integers of the form $r_0 = 4^m, r_1 = 4^{m+1}, \ldots, r_N = 4^{m+N}$. We assume it has at least two elements ($R^\eps \ge 4^3$), since for small $R$ the statement of the theorem holds vacuously. Observe that $4^N = \frac{r_N}{r_0} > \frac{R^\eps}{16}$. If for some $1 \le i\le N$ the ball $B(x, r_i)$ obeys the estimate
\[
h(x,r_i) \le \frac{r_i}{R^{1-\eps}} h(x,r_i/4),
\]
then we found a good ball, so assume these estimates all fail. If $h(x,r_0) = 0$, then $B(x, r_0)$ is good. If $h(x,r_0) \ge 1$, we have the following: 
\begin{align*}
R^{1/\eps} &> \Area H \ge h(x, r_N) \\
&> 4^N h(x,r_{N-1}) \\
&> 4^N 4^{N-1} h(x,r_{N-2}) \\ 
&> \ldots \\
&> 4^N 4^{N-1} \ldots 4^1 h(x,r_0) \\ 
&\ge \left(4^N\right)^{(N+1)/2} \\
&> \left(\frac{R^\eps}{16}\right)^{(\eps \log_4 R -1)/2}.
\end{align*}
This inequality fails for all large $R\ge R_0(\eps)$, but for all $R \le R_0(\eps)$ the statement of the theorem holds trivially.
\end{proof}

\begin{remark}
A careful analysis of this proof shows that the growth rate $V^{1+\eps}$ can be replaced by 
%$V \exp((\log V)^{\gamma})$ for any $\gamma \in (1/\sqrt{2}, 1)$. 
$V \exp(2(\log V)^{1/\sqrt{2}})$.
This function grows slower than any power $V^{1+\eps}$. Unfortunately, this proof cannot be refined to give the growth rate $V \polylog V$, as it is tempting to conjecture.
\end{remark}

In fact, Theorem~\ref{thm:sys-discrete} holds (with the same proof) for any coefficient ring $A$, not just $\mathbb{Z}/2$, if we define the systole $\sys^\alpha(M)$ as the shortest edge-length of a loop detected by $\alpha \in H^1(M; A)$, and define $\cut^{\alpha}(M)$ in the same way as with $\mathbb{Z}/2$ coefficients. For example, if $A = \mathbb{Z}/p$, then $\sys^\alpha$ might capture some loops not detected by the $\mathbb{Z}/2$-cohomology, but in return, the cutting complexes from the definition of $\cut^\alpha$ will have to cut those loops as well. However, we need to assume $A = \mathbb{Z}/2$ in order to relate $\cut^\alpha$ to the systolic invariants (and finish the proof of Theorem~\ref{thm:sys-first}). We do it for general $k$, as we will need it in the next section. 

\begin{lemma}[Minimality forces regularity]
\label{lem:regularity}
Let $\alpha \in H^k(M; \mathbb{Z}/2)$ be a non-trivial cohomology class, and let $H$ be an $(n-k)$-dimensional subcomplex in $M$ on which $\cut^\alpha(M)$ is attained. If $M$ is a manifold, then the top-dimensional simplices of $H$ form an $(n-k)$-cycle representing $\alpha^*$. Hence, $\cut^\alpha(M) = \sys_{\alpha^*}(M)$.
\end{lemma}

\begin{proof} 
\begin{comment}
Alexey's version:
Start with any class $\beta_1 \in H^{n-1}(M; \mathbb{Z}/2)$ such that $\alpha \smile \beta_1 \neq 0$ (equivalently, $\beta_1(\alpha^*) \neq 0$). Applying the Lusternik--Schnirelmann lemma, one makes sure that $\beta_1\vert_H \neq 0$ and finds a cycle $z_1$ supported in $H$ such that $\beta_1([z_1]) \neq 0$. If $\alpha^* = [z_1]$, we are done. If not, find a class $\beta_2 \in H^{n-1}(M; \mathbb{Z}/2)$ such that $\beta_2(\alpha^*) \neq 0$, $\beta_2([z_1]) = 0$. This can be done since $H^{n-1}(M; \mathbb{Z}/2) \cong \Hom_{\mathbb{Z}/2}(H_{n-1}(M; \mathbb{Z}/2), \mathbb{Z}/2)$ by the universal coefficient theorem. Again, $\beta_2\vert_H \neq 0$ so there is a cycle $z_2$ supported in $H$ such that $\beta_2([z_2]) \neq 0$. In particular, $[z_2] \neq [z_1]$, and they form a two-dimensional subspace in $H_{n-1}(M; \mathbb{Z}/2)$. Proceed in the same fashion. Having linearly independent $[z_1], \ldots, [z_k] \in H_{n-1}(M; \mathbb{Z}/2)$, with $z_1, \ldots, z_k$ supported in $H$, either $\alpha^*$ lies in their span (and then we managed to represent $\alpha^*$ by a cycle supported in $H$) or we find $\beta_{k+1} \in H^{n-1}(M; \mathbb{Z}/2)$ such that $\beta_{k+1}(\alpha^*) \neq 0$, $\beta_{k+1}([z_1]) = 0, \ldots, \beta_{k+1}([z_k]) = 0$. In the latter case we use the Lusternik--Shnirelmann lemma to find a cycle $z_{k+1}$ supported in $H$ and detected by $\beta_{k+1}$. Then $z_{k+1} \notin \langle [z_1], \ldots, [z_k]\rangle$ and we move on until we catch $\alpha^*$ in the span of the $[z_i]$. 

Hannah's version:
\end{comment}
It suffices to show that $H$ contains an $(n-k)$-cycle representing $\alpha^*$; if so, then this cycle is among the subcomplexes defining $\cut^{\alpha}$, and so by minimality it is equal to $H$.  Let $i \co H \hookrightarrow M$ denote the inclusion, and suppose for the sake of contradiction that $i_*(H_{n-k}(H))$ does not contain $\alpha^*$.  Then there is some hyperplane of $H_{n-k}(M)$ that contains $i_*(H_{n-k}(H))$ but not $\alpha^*$.  By the universal coefficient theorem, this hyperplane can be expressed as pairing with some $\beta \in H^{n-k}(M)$, for which $\beta(\alpha^*) \neq 0$. In this case $\beta\vert_{H} = 0$ while $\alpha \smile \beta\neq 0$, contradicting the Lusternik--Shnirelmann lemma because we have assumed that $\alpha \vert_{M \setminus H} = 0$.
\end{proof}

Note that the same proof shows that for any field coefficients, there is a cycle representing $\alpha^*$ with support $H$.  However, unless the coefficients are $\mathbb{Z}/2$ we do not control the coefficient of each simplex in $H$.  Only with $\mathbb{Z}/2$ coefficients does it make sense to say that $H$ itself forms a cycle achieving $\sys_{\alpha^*}(M)$. This completes the proof of Theorem~\ref{thm:sys-first}.

\begin{example}
\label{ex:sys<cut}
We construct a triangulated manifold $M$ of dimension $n$ with a cohomology class $\alpha \in H^1(M; \mathbb{Z}/2)$ such that $\sys^{\beta}(M) < \cut^{\alpha}(M)$ for any class $\beta \in H^{n-1}(M; \mathbb{Z}/2)$ with $\alpha \smile \beta \neq 0$. 
Let $M_1$ and $M_2$ be large identical triangulated copies of $S^1 \times S^{n-1}$.  We form $M$ as the connected sum of $M_1$ and $M_2$ by removing an $n$-simplex from each and identifying their boundaries.  We assume $n \geq 3$, although the argument is easily adapted for $n=2$.  For $n \geq 3$, the Mayer--Vietoris sequence implies 
\[H_{n-1}(M; \mathbb{Z}/2) \cong H_{n-1}(M_1; \mathbb{Z}/2) \oplus H_{n-1}(M_2; \mathbb{Z}/2) \cong \mathbb{Z}_2 \times \mathbb{Z}_2.\]
We select $\alpha \in H^1(M; \mathbb{Z}/2)$ such that $\alpha^*$ is the class $(1, 1)$ in $H_{n-1}(M; \mathbb{Z}_2)$.  Given any cycle representing $\alpha^*$, we may split it into an $M_1$ part and an $M_2$ part and add faces from the common boundary of the $n$-simplex to each part, to obtain a nontrivial cycle on each of $M_1$ and $M_2$.  Thus we have
\[\cut^{\alpha}(M) = \sys_{\alpha^*}(M) \geq 2\sys_{n-1}(M_1) - 2(n+1).\]
On the other hand, for any $\beta \in H^{n-1}$ such that $\beta$ pairs nontrivially with $\alpha^* = (1, 1)$, either $\beta$ pairs nontrivially with $(1, 0)$ or $\beta$ pairs nontrivially with $(0, 1)$.  Each of $(1, 0)$ and $(0, 1)$ in $H_{n-1}(M; \mathbb{Z}/2)$ can be represented by a cycle of size $\sys_{n-1}(M_1)$, so we have
\[\sys^{\beta}(M) \leq \sys_{n-1}(M_1) < \cut^{\alpha}(M).\]
%We take three cylindrical tubes $S_i \times [0,10]$, $S_i \cong S^{n-1}$, $i\in \{0,1,2\}$, with the metrics on their $S^{n-1}$-parts as follows: $S_1$ and $S_2$ are just round spheres of area $1$; $S_0$ is the connected sum of $S_1$ and $S_2$ of area $2$. We glue the tubes as in Figure~\ref{fig:sys<cut}, slightly modifying the metrics near the boundaries, to obtain $M$. The class $\alpha$ is such that the dual class $\alpha^*$ is represented by $S_0$. It can be seen that any class $\beta \in H^{n-1}(M; \mathbb{Z}/2)$ that intersects $\alpha$ has the systole attained at $S_1$ or $S_2$. Therefore, $1 = \sys^{\beta}(M) < \cut^{\alpha}(M) = 2$.
%\textcolor{red}{ For example, suppose that $M$ is the union of two manifolds $M_1$ and $M_2$ glued along a set of codimension at least $2$.  We can choose $\alpha$ such that it has nonzero restriction to each of $M_1$ and $M_2$.  Then $H$ is the union of a cycle representing $(\alpha\vert_{M_1})^*$ and a cycle representing $(\alpha\vert_{M_2})^*$, but any given $\beta$ detects one or the other of these cycles alone, so $\sup_{\beta} \sys^{\beta}(M) < \cut^{\alpha}(M)$.}
% \begin{figure}[h]
%     \centering
%     \includegraphics[width=0.4\textwidth]{syscut.png}
%     \caption{A manifold with a class $\alpha \in H^1$ such that $\cut^{\alpha} > \sup \{ \sys^{\beta} ~\vert~ \alpha \smile \beta \neq 0\}$.}
%     \label{fig:sys<cut}
% \end{figure}

\end{example}

\section{Generalizations}
\label{sec:generalize}

\begin{comment}
The proof above carries over almost \emph{verbatim} to the case when $M$ is a pure simplicial $n$-complex\footnote{In a \emph{pure} simplicial complex, every simplex of any dimension is incident to a top-dimensional simplex.}, rather than a manifold. Lemma~\ref{lem:coarea} still holds in that weak version. The Poincar\'e duality is no longer valid, so we must pick $H$ in a different way. We proceed using a version of the trick by Papasoglu~\cite{papasoglu2020uryson}. Take $H$ to be a minimal $(n-1)$-subcomplex of $M$ such that $\alpha$ vanishes when restricted to $M \setminus H$. The argument goes through, as the cut and paste trick works in this case as well: $h(x,t) \le \Area S(x,t)$ as otherwise one can cut $H \cap B(x,r)$ out of $H$ and paste $S(x,t)$, decreasing the area and preserving the $\alpha$-vanishing condition.
This proof gives $\sys^\alpha \cdot \Area H \le c V^{1+\eps}$. If we additionally know that there is a non-trivial $(n-1)$-cycle inside $H$, then we get a systolic estimate. This can be achieved by assuming that there is a cohomology class $\beta \in H^{n-1}(M; \mathbb{Z}/2)$ such that $\alpha \smile \beta \neq 0$. The Lusternik--Schnirelmann lemma applies: since $\alpha$ vanishes on the complement of $H$, $\beta$ restricts to $H$ non-trivially. Therefore, there is a non-trivial $(n-1)$-cycle in $H$ detected by $\beta$, and we proved that
\[
\sys^{\alpha}(M) \cdot \sys^{\beta}(M) \le c \Vol(M)^{1+\eps}.
\]
\end{comment}

The following result follows from Theorem~\ref{thm:sys-discrete} by induction.

\begin{theorem}
\label{thm:sys-induction}
For each $\eps > 0$, there is a constant $C = C(n,\eps)$ such that the following holds. Let $M$ be a simplicial $n$-complex with cohomology classes $\alpha_1, \ldots, \alpha_k \in H^1(M; \mathbb{Z}/2)$ whose cup-product is non-zero. Then
\[
\sys^{\alpha_1}(M) \cdot \ldots \cdot \sys^{\alpha_k}(M) \cdot \cut^{\alpha_1\smile \ldots \smile \alpha_k}(M) \le C \Vol(M)^{1+\eps}.
\]
\end{theorem}

We remark that Lemma~\ref{lem:regularity} applies and gives, in the case when $M$ is a manifold, 
\[
\cut^{\alpha_1\smile \ldots \smile \alpha_k}(M) =\sys_{(\alpha_1\smile \ldots \smile \alpha_k)^*}(M).
\]

Results very close to Theorem~\ref{thm:sys-induction} for $k=n-1$ have been known even with $\eps=0$ (cf.~\cite[Theorem~3]{guth2010systolic},  \cite[Theorem~1]{balacheff2016length}, \cite[Theorem~1.18]{avvakumov2021systolic}). For $k\le n-2$ one cannot plug $\eps=0$. 

\begin{proof}[Proof of Theorem~\ref{thm:sys-induction}] Assume $\eps < 1$.
Take $H$ to be a minimal $(n-1)$-subcomplex of $M$ such that $\alpha_1$ vanishes when restricted to $M \setminus H$. The proof in Section~\ref{sec:proof} implies that
\[
\sys^{\alpha_1}(M) \cdot \Area(H) \le c(n,\eps/3)\Vol(M)^{1+\eps/3}.
\]

The Lusternik--Schnirelmann lemma tells us that the product of the remaining $\alpha_i$ is non-trivial on $H$; hence, the induction assumption applies for $H$ in place of $M$. The systoles in $H$ dominate the corresponding systoles in $M$, since the intrinsic distances in $H$ dominate the extrinsic ones: $\sys^{\alpha_i}(M) \le \sys^{\alpha_i\vert_H}(H)$, $2 \le i \le k$. Any $(n-k)$-dimensional subcomplex of $H$ that cuts $\alpha_2\smile\ldots\smile\alpha_k\vert_H$ also cuts $\alpha_1\smile\ldots\smile\alpha_k$ in $M$, as one can deduce from the Lusternik--Schnirelmann lemma, so $\cut^{\alpha_1\smile \ldots \smile \alpha_k}(M) \le \cut^{\alpha_2\smile \ldots \smile \alpha_k\vert_H}(H)$.
Assembling this all together, we get
\begin{align*}
\sys^{\alpha_1}(M) &\cdot \sys^{\alpha_2}(M) \cdot \ldots \cdot \sys^{\alpha_k}(M) \cdot \cut^{\alpha_1\smile \ldots \smile \alpha_k}(M) \\
&\le \sys^{\alpha_1}(M) \cdot \sys^{\alpha_2\vert_H}(H) \cdot \ldots \cdot \sys^{\alpha_k\vert_H}(H) \cdot \cut^{\alpha_2\smile \ldots \smile \alpha_k\vert_H}(H) \\
&\le \sys^{\alpha_1}(M) \cdot C(n-1,\eps/3) \Area(H)^{1+\eps/3} \\
&\le C(n-1,\eps/3)^{1+\eps/3} (\sys^{\alpha_1}(M) \cdot \Area(H))^{1+\eps/3} \\ 
&\le  C(n-1,\eps/3)^{1+\eps/3}  (c(n,\eps/3) \Vol(M)^{1+\eps/3})^{1+\eps/3} \\
&\le C(n,\eps) \Vol(M)^{1+\eps}.
\end{align*}
\end{proof}

Another version of the same theorem holds in the continuous setting, with the assumption of locally bounded geometry replaced by its macroscopic cousin. We work in the setting of \emph{Riemannian polyhedra}, that is, finite pure\footnote{In a \emph{pure} simplicial $n$-complex, a simplex of any dimension is contained in an $n$-dimensional simplex.} simplicial complexes whose top-dimensional faces are endowed with Riemannian metrics matching on the common faces. We say that a subspace $H \subset M$ is a $k$-dimensional \emph{subpolyhedron} if it admits the structure of a finite pure $k$-dimensional simplicial complex whose cells are embedded smoothly in the cells of $M$. The piecewise Riemannian metric is inherited from the ambient polyhedron, allowing one to measure intrinsic distances and volumes. Any subpolyhedron $H$ admits a thin tubular neighborhood that deformation retracts onto $H$, and the Lusternik--Schnirelmann lemma for cohomology is applicable in the following form: if $\alpha\vert_{M\setminus H} = 0$, $\beta\vert_H = 0$, then $\alpha \smile \beta$ vanishes as a cohomology class of $M$. The systole $\sys^\alpha(M)$ is the infimum of the $k$-volumes of all (piecewise smooth or Lipschitz) $k$-cycles detected by $\alpha \in  H_k(M; \mathbb{Z}/2)$. The cutting area $\cut^\alpha(M)$ is the infimum of the $k$-volumes of all those $k$-dimensional subpolyhedra $H \subset M$ that cut $\alpha$ in the sense that $\alpha$ vanishes on $M\setminus H$.

\begin{theorem}
\label{thm:sys-macroscopic}
For each $\eps, v_1, v_2 > 0$, there is a constant $C = C(\eps, v_1,v_2)$ such that the following holds. Let $M$ be a compact Riemannian $n$-polyhedron such that $\sys_1(M) \ge 2$, every ball of radius $1/4$ has volume at least $v_1$, and every ball of radius $1$ has volume at most $v_2$. Let $\alpha_1, \ldots, \alpha_k \in H^1(M; \mathbb{Z}/2)$ be cohomology classes with non-zero cup-product. Then
\[
\sys^{\alpha_1}(M) \cdot \ldots \cdot \sys^{\alpha_k}(M) \cdot \cut^{\alpha_1\smile \ldots \smile \alpha_k}(M) \le C \Vol(M)^{1+\eps}.
\]
\end{theorem}
%\textcolor{blue}{Instead of $v_2$ we can require $\frac{v(x,1)}{v(x,1/4)} \le \beta$ for all $x$.}

Note that this theorem implies Theorem~\ref{thm:sys-macrocontinuous}.

\begin{proof}[Proof of Theorem~\ref{thm:sys-macroscopic}]
It suffices to prove this for $k=1$ and then induct. In the case $k=1$, let $\alpha = \alpha_1$, $R = \sys^\alpha (M)$, $V = \Vol M$, and proceed along the same scheme as in the proof of Theorem~\ref{thm:sys-discrete}. Trivially, $2v_1 \lfloor R\rfloor \le V$ (think of a necklace made of beads of radius $1/4$ along the systole), so $R \lesssim V$, where $\lesssim$ indicates an inequality that holds up to a factor depending on $\eps, v_1, v_2$. Let $A$ be the infimum of areas of $(n-1)$-subpolyhedra such that $\alpha$ vanishes on their complement. We need to show that $R \cdot A \lesssim V^{1+\eps}$. If $A \ge R^{1/\eps}$, then, as before,
\[
R \cdot A \le A^{1+\eps} \lesssim V^{1+\eps}, 
\]
but the last inequality requires an explanation. Pick a set of disjoint balls of radius $1/4$ such that the concentric balls of radius $3/4$ cover $M$, and for each of these find a concentric sphere of radius between $3/4$ and $1$ that has area at most $4v_2$ (by the co-area inequality). The union of these spheres has area at most $\frac{4v_2 V}{v_1}$, and $\alpha$ vanishes on the complement of this union by the curve factoring lemma (here we use the assumption $\sys_1(M) \ge 2$). Therefore, $A \le \frac{4v_2 V}{v_1}$. 

From now on assume that $A < R^{1/\eps}$.
Let $H$ be an almost-minimizing subpolyhedron for which $\alpha$ vanishes on its complement, with $A \le \Area H < A + \delta$, where $\delta = \frac{v_1}{R}$. 

%For brevity, we write $v(x,r) = \Vol B(x,r)$ and $h(x,r) = \Area(H \cap B(x,r))$.

A ball centered at $x \in H$ of radius $r \in [1,R]$ will be called \emph{good} if 
\[
h(x,r) \le \frac{13}{R^{1-\eps}} v(x,r/3),
\]
where we use again notation $v(x,r) = \Vol B(x,r)$ and $h(x,r) = \Area(H \cap B(x,r))$.

It suffices to show that every $x \in H$ is the center of a good ball. Once that is done, then we pick a Vitali cover of $H$ by good balls $B(x_i, r_i)$, such that the balls $B(x_i, r_i/3)$ do not intersect, and conclude just like in Section~\ref{sec:proof}:
\[
R \cdot A \le R \sum h(x_i, r_i) \le 13 R^\eps \sum v(x,r_i/3) \le 13 R^\eps V \lesssim V^{1+\eps}.
\]

We start by looking at balls $B(x,R)$, $B(x,R/4)$, $\ldots$, $B(x,R/4^N)$, where $N = \lfloor \eps \log_4 R \rfloor$ is the integer defined by $R/4^{N+1} < R^{1-\eps} \le R/4^N$, or equivalently, $4^N \le R^\eps < 4^{N+1}$ (note that $N \ge 0$ since $R > 1$).
If at least one of these balls obeys the inequality
\[
h(x,r) \le \frac{r}{R^{1-\eps}} h(x,r/4),
\]
then it is good, since
\begin{align*}
\frac{r}{R^{1-\eps}} h(x,r/4) &\le \frac{12}{R^{1-\eps}} \int_{r/4}^{r/3} (\Area S(x,t) + \delta) \; dt \\
&\le \frac{12}{R^{1-\eps}} v(x,r/3) + R^\eps \delta \\
&\le \frac{13}{R^{1-\eps}} v(x,r/3).    
\end{align*}

If none of them are good, then we have for each $0\le m \le N-1$:
\[
h(x,R/4^m) > \frac{R/4^m}{R^{1-\eps}} h(x,R/4^{m+1}) \ge 4^{N-m} h(x,R/4^{m+1}),
\]
and assembling these, 
\begin{align*}
    R^{1/\eps} &> A \ge h(x, R) \\
    &> 4^N h(x,R/4)  \\
    &> \ldots \\
    &> 4^N \ldots 4^1 h(x,R/4^N)  \\
    &\ge (4^{N+1})^{\frac{N}{2}} h(x,R^{1-\eps}) \\
    &> R^{\eps(\eps \log_4 R - 1)/2} h(x,1).
\end{align*}

The function $R^{\eps(\eps \log_4 R - 1)/2}$ grows (in $R$) faster than $R^{1-\eps + 1/\eps}$, so for $R \ge R_0(\eps, v_1)$ large enough we have
\[
h(x,1) < R^{1/\eps - \eps(\eps \log_4 R - 1)/2} < \frac{13v_1}{R^{1-\eps}} \le \frac{13 v(x,1/3) }{R^{1-\eps}},
\]
proving that the ball $B(x, 1)$ is good. For $R < R_0(\eps, v_1)$ the statement of the theorem follows from the above-mentioned observation $A \lesssim V$.
\end{proof}

\bibliography{sys}
\bibliographystyle{amsalpha}
\end{document}